\documentclass[12pt]{article}

\usepackage{amssymb,amsmath,amsfonts,amsthm,graphicx,subfigure,color,multicol}
\usepackage[margin=1.0in]{geometry}
\usepackage{empheq}
\usepackage{amssymb}
\usepackage{mathtools}
\usepackage{bbm}
\usepackage{dsfont}
\usepackage{epstopdf}
\usepackage{hyperref}
\usepackage{algpseudocode}
\usepackage{algorithm}

\newtheorem{thm}{Theorem}[section]
\newtheorem{lem}[thm]{Lemma}
\newtheorem{cor}[thm]{Corollary}

\newtheorem{dfn}[thm]{Definition}
\newtheorem{rem}[thm]{Remark}
\newtheorem{exa}[thm]{Example}

\usepackage{hyperref}
\usepackage{amssymb,amsmath,amsfonts,graphicx,color,multicol}
\newcommand{\scm}{\text{SCM}}
\newcommand{\scb}{\text{SCB}}
\newcommand{\gmo}{\gamma_\mathrm{sup,1}}
\newcommand{\gmtw}{\gamma_\mathrm{sup,2}}
\newcommand{\gmth}{\gamma_\mathrm{sup,3}}
\newcommand{\gmfo}{\gamma_\mathrm{sup,4}}
\newcommand{\gmfi}{\gamma_\mathrm{sup,5}}
\newcommand{\gmsi}{\gamma_\mathrm{sup,6}}
\newcommand{\gam}{\gamma}
\newcommand{\gm}{\gamma_\mathrm{sup}}
\newcommand{\gmk}{\gamma_{\mathrm{sup},k}}
\newcommand{\Ss}{{\mathcal{S}}}
\newcommand{\Pp}{{\mathcal{P}}}
\newcommand{\vr}{\varrho}
\newcommand{\noscm}{\nexists \mathrm{\  SCM}>0}
\newcommand{\noscb}{\nexists \mathrm{\  SCB}>0}
\newcommand{\isscm}{\exists \mathrm{\  SCM}>0}
\newcommand{\isscb}{\exists \mathrm{\  SCB}>0}

\begin{document}

\title{Exact optimal values of step-size coefficients for boundedness of linear
multistep methods}

\author{Lajos L\'oczi\thanks{{\texttt{lloczi@cs.elte.hu}}\newline
\indent \ This work was supported by the King Abdullah University of Science and Technology (KAUST), 4700 Thuwal, 23955-6900, Saudi Arabia. The author was also supported by the 
Department of Numerical Analysis, E\"otv\"os Lor\'and University (ELTE), and the 
Department of Differential Equations, Budapest University of Technology and Economics (BME), Hungary.}}

\date{\today}

\maketitle

\begin{abstract}
Linear multistep methods (LMMs) applied to approximate the
solution of initial value problems---typically arising from 
method-of-lines semidiscretizations of partial differential 
equations---are often required to have certain monotonicity or boundedness properties
(e.g.~strong-stability-preserving,  total-variation-diminishing or total-variation-boundedness properties).
These properties can be guaranteed by imposing step-size restrictions on the methods.
To qualitatively describe the step-size restrictions, one introduces the 
concept of step-size coefficient for monotonicity (SCM, 
also referred to as the
strong-stability-preserving (SSP) coefficient) or its generalization, 
the step-size coefficient for boundedness (SCB).  
A LMM with larger SCM or SCB is more efficient, and
the computation of the maximum SCM for a particular LMM is now straightforward.
However, it is more challenging to decide whether a positive SCB exists, or determine if a given 
positive number is a SCB.
Theorems involving sign conditions on certain 
linear recursions 
associated to the LMM have been proposed in the literature that allow us 
to answer the above questions:  
the difficulty with these theorems is that there are in general infinitely many sign conditions to be verified.
In this work we present methods to rigorously check the sign conditions.
As an illustration, we confirm some recent numerical investigations concerning the
existence of positive SCBs in the BDF and in the extrapolated BDF (EBDF) families.  
As a stronger result, we determine the optimal values of the SCBs 
as exact algebraic numbers in the BDF family (with $1\le k\le 6$ steps)
and in the Adams--Bashforth family (with $1\le k\le 3$ steps).

\smallskip

\end{abstract}

\section{Introduction}

Let us consider an initial-value problem
\begin{equation}\label{ODE}
u'(t)=F(u(t))\text{ for } t\ge 0,\text{ with } u(0)=u_0,
\end{equation}
where $F:\mathbb{V}\to\mathbb{V}$ is a given function, 
$u_0\in \mathbb{V}$ is a given initial value in some vector space $\mathbb{V}$, and $u$ denotes the unknown function. In applications it is often
crucial for the numerical solution $u_n$ to satisfy certain monotonicity or boundedness 
properties. 
\begin{exa}
Many important partial differential equations have the property that they preserve
\begin{itemize}
\item[(\textit{i})] the interval containing the initial data;
\item[(\textit{ii})] or, as a special case, non-negativity of the initial data.
\end{itemize} 
For example, if one considers a scalar hyperbolic conservation law with initial condition 
$U(x,t_0)\in [U_{\min},U_{\max}]$ with some constants 
$U_{\min}\le U_{\max}$ for $x\in\mathbb{R}$, then it is known that the solution satisfies $U(x,t)\in [U_{\min},U_{\max}]$
for $x\in\mathbb{R}$ and $t\ge t_0$. To approximate the solution $U$ of this partial differential equation, one
often uses a method-of-lines semidiscretization in space, and obtains a system of ordinary differential equations \eqref{ODE}. For many semidiscretizations, the initial-value problem \eqref{ODE} also preserves
(\textit{i}) or (\textit{ii}). Finally, one typically uses a Runge--Kutta method or a linear multistep method to discretize \eqref{ODE}: in this setting it is natural to require that the time discretization $u_n$ should also preserve (\textit{i}) or (\textit{ii}).
\end{exa}
In situations when the numerical method is a linear multistep method (LMM) approximating the solution of \eqref{ODE}, the boundedness property can be expressed as
\begin{equation}\label{boundednessproperty}
\|u_n\|\le \mu\cdot \max_{0\le j\le k-1}\|u_j\|\quad\quad (n\ge k),
\end{equation}
where the constant $\mu\ge 1$ is independent of $n$, the starting vectors $u_j$ ($0\le j\le k-1$) and the problem \eqref{ODE}; $\mu$ is determined only by the LMM. The monotonicity property, or strong-stability-preserving (SSP) property, is 
recovered if \eqref{boundednessproperty} holds with $\mu=1$. Common choices for the seminorm $\|\cdot\|$ on $\mathbb{V}$ in 
applications include the supremum norm or 
the total variation seminorm. For LMMs, a more detailed exposition of the above topics 
 together with references can be found, for example, in \cite[Section 1]{spijker2013}.
For Runge--Kutta methods, analogous questions have been analyzed thoroughly 
and solved satisfactorily in \cite{hundsdorferspijker2011}. In what follows, we focus on LMMs.

In the literature a considerable amount of work has been done on developing conditions
that guarantee \eqref{boundednessproperty}. One possibility is to impose some 
restrictions on the step size $\Delta t$ of the LMM. These restrictions lead to the concepts
of step-size coefficient for monotonicity (SCM) and  step-size coefficient for boundedness (SCB)---see
Definitions \ref{scbdef} and \ref{scmdef} below.
Depending on the context, the SCM is also referred to as the strong-stability-preserving (SSP) coefficient.
The SCB is a generalization of the SCM: for many practically important LMMs, there is 
no positive SCM, while a positive SCB still exists. It is thus natural to ask whether
a positive step-size coefficient (SCM or SCB) exists for a particular LMM, or determine if a given positive number is a step-size coefficient. Since a LMM with larger step-size coefficient is more efficient,
one is also interested in the maximum value of the SCM or SCB.
Conditions that are easy to check and are necessary and sufficient for the existence of a positive SCM,
or for a given positive number to be a SCM have already been devised, see \cite[Section 1.1]{spijker2013}. 

However, even for a single LMM, it seems more difficult  
\begin{itemize}
\item[(\textit{i})] to decide whether a positive \scb\ exists;
\item[(\textit{ii})] to determine if a given positive number is a \scb;
\item[(\textit{iii})] to compute the maximum \scb.
\end{itemize} 
In the rest of the paper, we pursue these goals. 
The theoretical framework we use is presented in \cite{hundsdorferspijkermozartova2012,spijker2013}, while the computational techniques we apply show many similarities with those of \cite{locziketcheson2014}. All computations in this work have been performed by using \textit{Mathematica} 10.

The structure of our paper is as follows. In Section \ref{preliminariessection} we present
some definitions and notation. In Sections \ref{section12} and \ref{section13} we review
the main results of \cite{hundsdorferspijkermozartova2012} and \cite{spijker2013} concerning  (\textit{ii}) and (\textit{i}) above, respectively. Section \ref{section2} contains our theorems
for three families of multistep methods: 
\begin{itemize}
\item for the extrapolated BDF (EBDF) methods we answer (\textit{i});
\item for the BDF methods (as implicit methods) we answer (\textit{iii});
\item for the Adams--Bashforth (AB) methods (as explicit methods) we answer (\textit{iii}). 
\end{itemize}
The proofs are described in Section \ref{section3}.

\begin{rem}
In our proofs we essentially need to establish the non-negativity of certain (parametric) linear recursions.
 Recently, some general results have been
devised solving the problem of (ultimate) positivity in several classes of integer linear recursions, see, for example, the series of papers \cite{survey,soda,book1,book2}. 
\end{rem}

\subsection{Preliminaries and notation}\label{preliminariessection}

A LMM has the form
\begin{equation}\label{LMMform}
u_n=\sum_{j=1}^k a_j u_{n-j}+\Delta t\sum_{j=0}^k b_j F(u_{n-j}) \quad\quad (n\ge k),
\end{equation}
where $k\ge 1$, the step number of the LMM, is a fixed integer, and the coefficients $a_j, b_j\in\mathbb{R}$  determine the method. The step size of the method $\Delta t>0$ is assumed
to be fixed,
and we suppose that the starting values for the LMM, 
$u_0$ (appearing in \eqref{ODE}) and $u_j$ ($1\le j\le k-1$), are also given.
The quantity $u_n$ approximates the exact solution value $u(n\Delta t)$.
The generating polynomials associated with the LMM are denoted by
\begin{equation}\label{rhosigmadef}
\rho(\zeta):=\zeta^k-\sum_{j=1}^k a_j \zeta^{k-j}\quad\text{and}\quad
\sigma(\zeta):=\sum_{j=0}^k b_j \zeta^{k-j}.
\end{equation}
A non-constant univariate polynomial is said to satisfy the \textit{root condition}, if all of its roots
have absolute value $\le 1$, and any root with absolute value $=1$ has multiplicity one.
As in \cite{spijker2013}, the LMMs in this work are also required to satisfy the following basic assumptions.
\begin{equation}\label{LMMbasicassumptions}
\begin{aligned}
1.\quad  & \sum_{j=1}^k a_j=1\text{ and } \sum_{j=1}^k j a_j=\sum_{j=0}^k b_j  & \text{ (consistency).}\\
2.\quad &  \text{The polynomial } \rho \text{ satisfies the root condition} & \text{(zero-stability).}\\
3.\quad &  \text{The polynomials } \rho \text{ and } \sigma 
\text{ have no common root} & \text{(irreducibility).}\\
4.\quad &  b_0\ge 0. & 
\end{aligned}
\end{equation}
All well-known methods used in practice satisfy the four conditions in \eqref{LMMbasicassumptions}.

The \textit{stability region} of the LMM, denoted by $\Ss$, is defined as
\[
\Ss:=\{ \lambda\in\mathbb{C} : 1-\lambda b_0\ne 0\text{ and } \rho-\lambda\sigma
\text{ satisfies the root condition}\},
\]
see \cite[Section 2.1]{spijker2013}. 
The interior of the stability region will be denoted by $\mathrm{int}(\Ss)$.
\begin{rem}
Notice that the above definition of the stability region $\Ss$ is slightly more restrictive than the usual
one. The usual definition of the stability region (see, for example, in \cite{hairerwanner}),
\[{\widetilde{\mathcal{S}}}:=\{ \lambda\in\mathbb{C} : \rho-\lambda\sigma
\text{ satisfies the root condition}\},
\]
does \emph{not} exclude the case of a vanishing leading coefficient of the polynomial 
${\cal{P}}(\cdot,\lambda):=\rho(\cdot)-\lambda \sigma(\cdot)$. With this definition 
${\widetilde{\mathcal{S}}}$, one can construct simple examples with the following properties:
\begin{itemize}
\item the order of the recurrence relation generated by the LMM becomes $<k$ for certain values of the step size $\Delta t>0$, hence $k$ starting values of the LMM cannot be chosen arbitrarily;
\item there is an isolated point of the boundary of ${\widetilde{\mathcal{S}}}$ (being an element of ${\widetilde{\mathcal{S}}}$);
\item the boundary of ${\widetilde{\mathcal{S}}}$ is not a subset of the root locus curve due to these isolated boundary points.
\end{itemize}
Similarly, in the class of multiderivative multistep methods (being a generalization of LMMs), it seems advantageous to exclude the values of $\lambda\in\mathbb{C}$ from the definition of the stability region for which the leading coefficient of the corresponding polynomial ${\cal{P}}(\cdot,\lambda)$ vanishes.
\end{rem}
The set of natural numbers $\{0,1,\ldots\}$ is denoted by $\mathbb{N}$, while the 
complex conjugate of $z$ is $\bar{z}$. The \textit{dominant root} of a non-constant univariate polynomial is any root having the largest absolute value. 

When we define algebraic numbers in later sections, a polynomial 
\[\sum_{j=0}^n a_j x^j \text{ with } a_j\in\mathbb{Z}, a_n\ne 0 \text{ and } n\ge 3\] will be represented simply by its coefficient list 
\begin{equation}\label{coeffdef}
\{a_n, a_{n-1},\ldots, a_0\}.
\end{equation}

Now we recall the definition of the \textit{step-size coefficient for boundedness} and \textit{monotonicity}, respectively, corresponding to 
a given linear multistep method.
\begin{dfn}\label{scbdef}
Suppose that the method coefficients $a_j\in\mathbb{R}$ ($1\le j\le k$) and $b_j\in\mathbb{R}$ ($0\le j \le k$) satisfy \eqref{LMMbasicassumptions}. We say that $\gam>0$ is a step-size coefficient for boundedness (\scb) of the corresponding LMM, if $\exists\ \mu \ge 1$ such that
\begin{itemize}
\item for any vector space with seminorm $(\mathbb{V},\|\cdot\|)$, 
\item for any function $F:\mathbb{V}\to\mathbb{V}$ satisfying 
\[\exists\tau >0\ \  \forall v\in \mathbb{V} \ :\  \|v+\tau F(v)\|\le \|v\|,\]
\item for any $\Delta t\in (0,\gam\,\tau]$,
\item and for any starting vectors $u_j\in \mathbb{V}$ ($0\le j\le k-1$),
\end{itemize}
the sequence $u_n$ generated by \eqref{LMMform} has the property
$\|u_n\|\le \mu\cdot \max_{0\le j\le k-1}\|u_j\|$ for all $n\ge k$.
\end{dfn}
\begin{dfn}\label{scmdef}
We say that $\gam>0$ is a step-size coefficient for monotonicity (\scm) of the LMM,
if Definition \ref{scbdef} holds with $\mu=1$.
\end{dfn}
\smallskip
\noindent Given a LMM, the following abbreviations will be used throughout this work:
\begin{itemize}
\item $\isscm$ and $\noscm$ to indicate that there is a positive / there is no positive step-size coefficient for monotonicity, respectively;
\item $\isscb$ and $\noscb$ to indicate that there is a positive / there is no positive step-size coefficient for boundedness, respectively.
\end{itemize}
\noindent It is clear from Definitions \ref{scbdef}-\ref{scmdef} that for a given LMM
\[
\isscm \implies \isscb.
\]
If $\isscb$, then we define
\[
\gm:=\sup\{\gamma>0 : \gamma \text{ is a } \scb \}.
\]
When a family of $k$-step LMMs is given, sometimes we will use the symbol $\gmk$ instead.

\subsection{A necessary and sufficient condition for $\gamma>0$ to be a $\scb$}\label{section12}

Let us fix a particular LMM. For a given $\gam\in\mathbb{R}$, we define an auxiliary sequence $\mu_n(\gam)$ ($n\in\mathbb{Z}$) 
as in \cite[(2.10)]{spijker2013} by
\begin{equation}\label{mundef}
\mu_n(\gam) :=\left\{
\begin{aligned}
&   0 & \text{ for } & n<0,\\
&   b_n-\gam\,b_0\mu_n(\gam)+\sum_{j=1}^k (a_j -\gam\,b_j)\mu_{n-j}(\gam) & \text{for } & 0\le n \le k,\\
& -\gam\,b_0\mu_n(\gam)+\sum_{j=1}^k (a_j -\gam\,b_j)\mu_{n-j}(\gam) &
 \text{for } & n>k. 
\end{aligned}
\right.
\end{equation}
The following characterization appears in \cite[Theorem 2.2]{spijker2013}.
\begin{thm}\label{thm1.1}
Suppose the LMM satisfies \eqref{LMMbasicassumptions} and let $\gam>0$ be given. Then  $\gam$ is a $\scb$ if and only if 
\begin{equation}\label{necsufcond2}
-\gam\in \mathrm{int}(\Ss), \text{ and } \mu_n(\gam)\ge 0 \text{ for all } n\in\mathbb{N}^+.
\end{equation}
\end{thm}
The above theorem is based on the material developed in 
\cite{hundsdorferspijkermozartova2012}. In \cite[Section 6]{hundsdorferspijkermozartova2012}, the authors numerically determine the maximum
\scb\ values for members of several parametric families of LMMs by repeatedly applying the following
test. For a particular LMM and given
$\gam>0$, they check if
$\gam$ is a $\scb$ by choosing a large $N\in\mathbb{N}$, and 
verifying $\mu_n(\gam)\ge 0$ for all $1\le n\le N$. However, as the authors point out in  \cite{hundsdorferspijkermozartova2012},  it is not obvious 
(neither \textit{a priori} nor \textit{a posteriori})
how large $N$ 
one should choose to conclude---with high 
certainty---that $\mu_n(\gam)\ge 0$ for all $n\in\mathbb{N}^+$. They typically
use $N\approx 10^3$; as a comparison, see our Remark \ref{bdf4remark}.

\subsection{The existence of a \scb}\label{section13}

For a fixed LMM and given $\gam>0$, Theorem \ref{thm1.1} provides a necessary and sufficient condition for $\gam$ to be a $\scb$. But to decide---with the 
help of this theorem---whether $\noscb$, one should check condition \eqref{necsufcond2} for infinitely many 
$\gam>0$ values, and for each $\gam$, there are infinitely many sign conditions
$\mu_n(\gam)\ge 0$ to be verified.

To overcome this difficulty, \cite[Theorem 3.1]{spijker2013}
combines Theorem \ref{thm1.1} with the results of \cite{tijdeman2013}
to present some \textit{simpler} conditions that are \textit{almost}
necessary and sufficient for $\isscb$. 
 ``Almost'' in the previous sentence means that the conditions in \cite[Theorem 3.1]{spijker2013} are 
necessary and sufficient for $\isscb$ (not in the full, but) in a slightly restricted
class of LMMs; and ``simpler'' means that these conditions do not involve the 
parametric recursion $\mu_n(\gam)$ in \eqref{mundef}, rather, a non-parametric recursion $\tau_n$
determined by the method coefficients as
\begin{equation}\label{taundef}
\tau_n :=\left\{
\begin{aligned}
&   0 & \text{ for } & n<0,\\
&   b_n+\sum_{j=1}^k a_j \tau_{n-j} & \text{for } & 0\le n \le k,\\
&  \sum_{j=1}^k a_j \tau_{n-j} &
 \text{for } & n>k.
\end{aligned}
\right.
\end{equation}
Since we will not work with \cite[Theorem 3.1]{spijker2013} directly, here we cite only   
\cite[Corollary 3.3]{spijker2013}. 
\begin{cor}\label{spijkercorollary3.3}
Suppose the LMM satisfies \eqref{LMMbasicassumptions}. We define 
\begin{equation}\label{n0def}
n_0:=\min\{ n: 1\le n\le k \text{ and } \tau_n\ne 0\}.
\end{equation}
\begin{itemize}
\item[(i)] If $\tau_n>0$ for all $n\ge n_0$, and the only root of the polynomial 
$\rho$ appearing in \eqref{rhosigmadef} with modulus $1$ is $1$, then $\isscb$.
\item[(ii)] If $\tau_n\le 0$ for some $n\ge n_0$ being a multiple of $n_0$, then 
$\noscb$.
\end{itemize}
\end{cor}
\noindent The index $n_0$ defined above can be shown to exist
due to consistency and zero-stability of the LMM.

As an application of \cite[Theorem 3.1]{spijker2013} or Corollary \ref{spijkercorollary3.3}, 
\cite[Section 5]{spijker2013} analyzes some well-known classical LMMs, including 
\begin{itemize}
\item the Adams--Moulton (or implicit Adams), 
\item the Adams--Bashforth (or explicit Adams), 
\item the BDF, 
\item the extrapolated BDF (EBDF), 
\item the Milne--Simpson and 
\item the Nystr\"om methods.
\end{itemize}
These investigations confirm and extend some earlier results 
\cite{hundsdorferruuthspiteri2003,hundsdorferruuth2006,hundsdorferspijkermozartova2011,hundsdorferspijkermozartova2012} concerning the existence of step-size coefficients for monotonicity
or step-size coefficients for boundedness.  
The results of \cite[Section 5]{spijker2013} have the following form. 

Consider a discrete family of LMMs from the previous paragraph, parametrized by the step number $k\in\mathbb{N}$. Let $1\le k_\text{min}\le k_\text{max}\le +\infty$ denote 
some fixed bounds on $k$ coming from practical considerations (e.g. zero-stability of the LMM), that is, we consider the step numbers   
$k_\text{min}\le k\le k_\text{max}$.
Then there exist two integers $0\le k_\text{mon} \le k_\text{bdd}$ such that
\begin{itemize}
\item[$\bullet$] $\isscm$ $\Longleftrightarrow k_\text{min}\le k\le k_\text{mon}$;
\item[$\bullet$] $(\noscm\text{ and }\isscb )\Longleftrightarrow k_\text{mon}+1\le k\le k_\text{bdd}$;
\item[$\bullet$] $\noscb \Longleftrightarrow    k_\text{bdd}+1\le k\le k_\text{max}$.
\end{itemize}
It is to be understood that if $\ell_1\le k\le \ell_2$ with $\ell_1>\ell_2$ in any of the inequalities above, then the corresponding case does not occur.
Some examples from \cite[Section 5]{spijker2013} are provided in the table below.
\begin{table}[H]
\centering
\begin{tabular}{|c||c|c|c|c|}
\hline
\textbf{LMM family} & $k_\text{min}$ & $k_\text{max}$  & $k_\text{mon}$ &  $k_\text{bdd}$ \\ \hline
Adams--Bashforth & 1 & $+\infty$ & 1 & 3 \\ \hline
BDF & 1 & 6 & 1 & 6 \\ \hline
EBDF & 1 & 6 & 1 &  5 \\ \hline
Milne--Simpson & 2 & $+\infty$ & 1 &  1 \\ \hline
\end{tabular}
\end{table}
Out of the several LMMs investigated in \cite[Section 5]{spijker2013}, there are
however two families---the BDF methods with $3\le k\le 6$ steps, and the EBDF methods with $3\le k\le 5$ steps---for which the corresponding inequalities 
\begin{equation}\label{taupositivity}
\tau_n>0\quad \text{for } n\ge n_0
\end{equation}
appearing in Corollary \ref{spijkercorollary3.3} are not verified completely.
More precisely, \eqref{taupositivity} is verified only up to a finite value $n_0\le n\le N$ (for example, up to $N=500$), and it is observed that, for these large $n$ values, $\tau_n$ is already close enough
to $\lim_{n\to+\infty}\tau_n=1$ to conclude (``we have no formal proof \ldots, but convincing numerical evidence instead'') the validity of \eqref{taupositivity}
(see \cite[Conclusions 5.3 and 5.4]{spijker2013}).

\section{Main results}\label{section2}

\subsection{Positivity of the $\tau_n$ sequences in the $\text{EBDF}$ family }

\begin{thm}\label{thm2.1}
Let us fix any $3\le k\le 5$ and consider the EBDF family with $k$ steps. Then the sequence $\tau_n$ satisfies
$\tau_n>0$ for $n\ge n_0=1$ (see  \eqref{taundef} and \eqref{n0def}).
\end{thm}
The above theorem completes and verifies 
the numerical proof of \cite[Conclusion 5.4]{spijker2013} regarding the EBDF methods with $k\in\{3,4,5\}$
steps. In the proof of Theorem \ref{thm2.1}, given in Sections \ref{proofsummaryEBDF} and \ref{EBDFmethodssection},  we explicitly represent  $\tau_n$ as a linear
combination of powers of algebraic numbers to estimate this sequence from below and 
hence prove its positivity. 

As a combination of \cite[Conclusion 5.4]{spijker2013} and our Theorem \ref{thm2.1} we obtain the following result.
\begin{cor}In the EBDF family
\begin{itemize}
\item $\isscm$ for the $1$-step EBDF method;
\item $\noscm$ but $\isscb$ for the $k$-step EBDF method with $k\in\{2,3,4,5\}$;
\item $\noscb$ for the $6$-step EBDF method.
\end{itemize}
\end{cor}

\subsection{Exact optimal \scb\  values in the BDF family}\label{section23exactoptimal}

We complete the numerical proof of \cite[Conclusion 5.3]{spijker2013} concerning the existence of \scb\ for the BDF methods with  $3\le k\le 6$ steps. 
However, instead of just proving the positivity of the corresponding sequences $\tau_n$, we directly
determine the exact and optimal values of the \scb\ constants for $2\le k\le 6$.
For the sake of completeness, the $k=1$ case (the implicit Euler method) is also included. 
The approximate numerical values of $\gmk$ below have been rounded down. The polynomial 
coefficients---see \eqref{coeffdef} for the notation---corresponding to the cases $k=5$ and $k=6$ have been aligned for easier readability (and they are to be read in the usual way, horizontally from left to right).

\begin{thm}\label{thm2.2}
 The optimal values of the step-size coefficients for boundedness $\gmk$ in the BDF
family are given by the following exact algebraic numbers:
\begin{itemize}

\item $\gmo=+\infty;$

\item $\gmtw=1/2;$

\item $\gmth\approx 0.831264155297$ is the smallest real root of the 4th-degree polynomial\\

{\centerline{\footnotesize{$\{5184, -539352, 4277340, -7093698, 3248425\};$}}}

\item $\gmfo\approx 0.486220284043$ is the unique real root of the 5th-degree polynomial\\

\centerline{\footnotesize{$\{147456, -4065024, 97751296, -178921248, 146499984, -39945535\};$}}

\item $\gmfi\approx 0.304213712525$ 
is the smaller real root of the 10th-degree polynomial
\begin{table}[H]
\center
\footnotesize{
\hskip0.6cm
\begin{tabular}{r r r}
\{9183300480000000000, & 85812841152000000000,  & 11922800956027200000000, \\ 
$-$158236459797931200000000, &  1300372831455671124000000, & $-$3469598208824475416400000, \\ 
5222219230639370911710000, & $-$4938342912266137089480000, & 2829602902356809601352800,\\
$-$897140360120473365541380, & 113406532200497326720157\};
\end{tabular}}
\end{table}
\item $\gmsi\approx 0.131359487166$ is the smaller real root of the 18th-degree polynomial 
\begin{table}[H]
\fontsize{8}{11}\selectfont
\center
\begin{tabular}{rr}
\{301499153838045275528311603200000000, & 122639585534504839818945201438720000000, \\
384963168041618344234237602954215424000000, & 27549570033081885223128023207444584857600000,\\
688321830171904949334479202088109368934400000, & $-$3841469418723966761157769983211793789485056000,\\
114843588487750902323103668249803599786305126400, & $-$1006269459507863531788997342497299304467812843520,\\
5587246198359348966734174906666273788289332150272, & $-$17429944795858965010882996868073155329514839408640, \\
35959114141443095864886240750517884787497897431040, & $-$53357827225132542443145327442029250536098863687680,\\
58779078470720235677143648519968524504336318905600, & $-$48117131040654192740877887801688549303578668712064, \\
28809153195856173726312967696976168633917662024240, & $-$12158530101520566099221248226347019432756062262240, \\
3383327891741061214240426918034255832010259451480, & $-$541370800878125712591610585145194659522378896880, \\
33328092641186254550760247661168148768262937067\}. &
\end{tabular}
\end{table}
\end{itemize}
\end{thm}
The proofs of the above results are given in Sections \ref{proofsummaryBDF} and \ref{BDFmethodssection}.
From a technical point of view, the proof of the $k=3$ case is different from the other cases,
see Remark \ref{remark32}.

\subsection{Exact optimal \scb\ values in the Adams--Bashforth family}

To further illustrate our techniques, we have computed the largest \scb\ 
values for an explicit LMM family
as well; we chose the Adams--Bashforth methods with $1\le k\le 4$ steps.

For $k=1$ (i.e.~for the explicit Euler method) 
it is known (\cite[Theorem 5.2]{spijker2013})
 that $\isscm$, hence $\isscb$. 

For any $k\ge 4$, \cite[Theorem 5.2]{spijker2013} proves---with the help of the
sequence $\tau_n$---that 
$\noscb$.
The reason we include the $k=4$ case here is to show an example of using
the parametric sequence $\mu_n(\gam)$ and Theorem \ref{thm1.1} instead of $\tau_n$ 
in Corollary \ref{spijkercorollary3.3} (\textit{ii}) to detect $\noscb$.

\begin{thm} The optimal values of the step-size coefficients for boundedness in the 
Adams--Bashforth family are given by the rational numbers below:
\begin{itemize}
\item $\gmo=1;$
\item $\gmtw=4/9\approx 0.44444;$
\item $\gmth=84/529\approx 0.15879;$
\item for $k=4$, $\noscb$.
\end{itemize}
\end{thm}
\noindent The proofs of these results are found in Sections \ref{proofsummaryAB} and \ref{ABmethodssection}.

\section{Proofs}\label{section3}

\subsection{Summary of the proof techniques for the EBDF methods}\label{proofsummaryEBDF}

The proofs in Section \ref{EBDFmethodssection} for the EBDF methods use the following argument.
 Since $\tau_n$ in \eqref{taundef} is a solution of a linear recursion, it is 
represented as
\begin{equation}\label{taurepresentation}
\tau_n=\sum_{j=1}^k c_j \varrho_{j}^n,
\end{equation}
where the quantities $\varrho_{j}\in\mathbb{C}$ 
are the roots of the corresponding characteristic polynomial 
(without multiple roots for each EBDF method), and the constants
$c_j\in\mathbb{C}$ are determined by the starting values. 
By bounding $|c_j|$ and $|\varrho_j|$, we prove the inequality $\tau_n>0$ for all $n\ge 1$.

\subsection{Summary of the proof techniques for the BDF methods}\label{proofsummaryBDF}

The proofs in Section \ref{BDFmethodssection} for the BDF methods are based on the following.
For any given $\gam>0$, the linear recursion \eqref{mundef} takes the form
\begin{equation}\label{genmunrec}
\sum_{j=0}^k c_{j}(\gamma)\mu_{n-j}(\gamma)=0\quad (k\le n\in\mathbb{N}),
\end{equation}
where the coefficients $c_{j}(\gam)$  ($0\le j\le k$) and the 
 starting values $\mu_j(\gam)$ ($0\le j\le k-1$) are determined by the LMM.
The corresponding
characteristic polynomial is denoted by
\begin{equation}\label{Ppkcharpoly}
\Pp_k(\vr,\gam):=\sum_{j=0}^k c_{j}(\gamma)\vr^{k-j}.
\end{equation}
We apply the characterization in Theorem \ref{thm1.1} together with Observations 1-4 presented below. 
Lemma \ref{upperestlemma1} and Lemma \ref{upperestlemma2} will be used to bound $\gm$ from above for the $k$-step BDF methods with $k=3$ and $k\in\{2, 4, 5, 6\}$, respectively. Then, 
by using representations similar to \eqref{taurepresentation} and Observation 4, we show in each case that the proposed upper bound for $\gm$ is sharp.
\bigskip
\bigskip

\noindent \quad $\bullet$ \textbf{Observation 1}

\noindent For a $k$-step BDF method ($1\le k\le 6$), it is known \cite{hairerwanner} that
$-\gam\in \mathrm{int}(\Ss)$ for any $\gam>0$. Therefore, 
the condition \eqref{necsufcond2} in Theorem \ref{thm1.1} reduces to 
$\mu_n(\gam)\ge 0$ ($n\in\mathbb{N}^+$).
\bigskip
\bigskip

\noindent \quad $\bullet$ \textbf{Observation 2}

\noindent It is easily seen from Definition \ref{scbdef} that if $\gam_0>0$ is a $\scb$, then 
each number from the interval $(0,\gam_0]$ is also a $\scb$; thus, by \eqref{necsufcond2}, we also have  
$\mu_n(\gam)\ge 0$ for all $n\in\mathbb{N}^+$ and $\gam\in(0,\gam_0]$.
Since the function $\gam\mapsto \mu_n(\gam)$ (clearly  
being a rational function for any fixed $n\in\mathbb{N}$ due to the form of the linear recursion \eqref{mundef}) cannot be non-negative in a neighborhood of a
simple zero, we immediately obtain the following upper bound on $\gm$ 
(in the lemma, $\mu'_n$ denotes the derivative of the function $\mu_n(\cdot)$). 
\begin{lem}\label{upperestlemma1}
Suppose  there exist some $n\in\mathbb{N}^+$ and $\gam^*>0$ such that $\mu_n(\gamma^*)=0$ and $\mu'_n(\gamma^*)\in\mathbb{R}\setminus\{0\}$. Then
$\gm\le\gam^*$.
\end{lem}
\bigskip
\smallskip
\noindent \quad $\bullet$ \textbf{Observation 3}

\noindent The following lemma will be applied to bound $\gm$ from above when the characteristic polynomial has a unique pair
of complex conjugate roots that are dominant.  
\begin{lem}\label{upperestlemma2}
Suppose that $z\in\mathbb{C}\setminus\mathbb{R}$ with $|z|=1$, $w\in\mathbb{C}\setminus\{0\}$, and a \emph{real} sequence $\nu_n\to 0$ ($n\to +\infty$) are given. Then
$w z^n+\bar{w} (\bar{z})^n+\nu_n<0$ for infinitely many $n\in\mathbb{N}$.
\end{lem}
\begin{proof} We introduce $\varphi, \psi\in[0,2\pi)$ via the relations $z=\exp(i\varphi)$ and $w=|w|\exp(i\psi)$. Due to symmetry, we can suppose that $\varphi\in(0,\pi)$, so there is a 
 $\delta\in(0,\pi/2)$ such that $\delta<\varphi<\pi-\delta$.
Then 
\begin{equation}\label{zwnun}
w z^n+\bar{w} (\bar{z})^n+\nu_n=2|w|\cos\left(n\varphi+\psi\right)+\nu_n.
\end{equation} 
We show that 
\begin{equation}\label{cosineq}
\cos\left(n\varphi+\psi\right)\le\cos(\pi/2+\delta/2) \text{ for infinitely many } n.
\end{equation} 
Indeed, the inequality in \eqref{cosineq} holds if and only if $n\in\mathbb{N}$ and $k\in\mathbb{Z}$
are chosen such that
\begin{equation}\label{LHSRHSineq}
\text{LHS}:=(\pi/2+\delta/2-\psi+2\pi k)/\varphi\le n \le(3\pi/2-\delta/2-\psi+2\pi k)/\varphi=:\text{RHS}.
\end{equation}
But $\text{RHS}-\text{LHS}=(\pi-\delta)/\varphi>1$, so, by taking 
$k\in\mathbb{N}$ larger and larger, we see that there are infinitely many
$n\in\mathbb{N}$ satisfying  \eqref{LHSRHSineq}.
Finally, by using $|w|\ne 0$, \eqref{cosineq}, $\cos(\pi/2+\delta/2)<0$ and $\nu_n\to 0$, we get that \eqref{zwnun} is also negative for infinitely many $n$ indices.
\end{proof}
\bigskip
\smallskip
\noindent \quad $\bullet$ \textbf{Observation 4}

\noindent By taking into account the first sentence of Observation 2, we get the following
lower bound.
\begin{equation}\label{sect32intro}
\exists\,\gam_0>0 :   \mu_n(\gam_0)\ge 0\  (\forall\,n\in\mathbb{N}^+)
\implies \gm\ge \gam_0.
\end{equation}

\begin{rem}
Notice the similarities between  Lemma \ref{upperestlemma1} and \cite[Lemma 4.5]{kraaijevanger},
and between Lemma \ref{upperestlemma2} and \cite[Lemma 3.1]{kraaijevanger}.
Also compare Lemma \ref{upperestlemma2} and \cite[Theorem 4.3]{spijker2013}.
\end{rem}
\begin{rem}\label{bdf4remark}
Obtaining the exact value of $\gmfo\approx 0.48622$ proved to be significantly harder
than determining that of $\gmth$,
because we could not apply Lemma \ref{upperestlemma1} to bound $\gmfo$
from above. 
The value of $\gmfo$ was found via a series of numerical experiments. For example, to see 
$\gmfo<0.48625$, one checks that the sequence $\mu_n$ in Theorem \ref{thm1.1} for $1\le n\le 27000$ satisfies
\[
\mu_n({48625/100000})<0 \Longleftrightarrow n\in\{ 26814, 26875, 26886, 26936, 26947, 26997\}.
\]
To find all these six indices, we used 16000 digits of 
precision to evaluate the terms of the recursion $\mu_n({48625/100000})$---15000 digits would be insufficient. In fact, these experiments led to the
formulation of Lemma \ref{upperestlemma2}.
\end{rem}
\begin{rem}\label{remark32}
Regarding the determination of $\gmth$, the characteristic polynomial 
$\Pp_3(\cdot,\gam)$ has one real root $\vr_1(\gam)>0$ and a pair of complex conjugate roots
  $\vr_{2,3}(\gam)$ with $|\vr_1(\gam)|=|\vr_2(\gam)|=|\vr_3(\gam)|$ for $\gam=5/6\approx 0.83333$. From Lemma \ref{upperestlemma2} we would get the bound
$\gmth\le 5/6$, but this bound is not sharp. However, Lemma \ref{upperestlemma1}
with $n=6$ yields the exact value of $\gmth\approx 0.83126$.
\end{rem}

\subsection{Summary of the proof techniques for the Adams--Bashforth methods}\label{proofsummaryAB}

Since these LMMs are explicit, we have $b_0=0$ in \eqref{LMMform}, so from \eqref{mundef} we see that for any $n\in\mathbb{N}$ the function $\gam\mapsto \mu_n(\gam)$ is a polynomial and $\mu_0(\gam)=0$.
For $1\le k\le 3$, we study the roots of these polynomials $\mu_n(\cdot)$ for small $n$
to conjecture the value of $\gmk$. 
Of course, Observation 1 from the previous section 
cannot be applied now, because we have to take into account
the condition $-\gamma\in \mathrm{int}(\Ss)$ in \eqref{necsufcond2} as well.
So we use Lemma \ref{upperestlemma1} together with 
\begin{equation}\label{17mod}
(\exists\,\gam_0>0 :   -\gam_0\in \mathrm{int}(\Ss) \text{ and } \mu_n(\gam_0)\ge 0\  (\forall\,n\in\mathbb{N}^+))
\implies \gm\ge \gam_0
\end{equation}
to verify that the conjectured $\gm$ is indeed the optimal \scb. 

\begin{rem}
For $2\le k\le 3$, it turns out that the dominant root of the characteristic polynomial 
$\Pp_k(\cdot,\gamma)$ in \eqref{Ppkcharpoly} is positive real for $\gamma=\gmk$, so in these cases 
 a result similar to Lemma \ref{upperestlemma2} is not applicable.
\end{rem}

\subsection{Proofs for the EBDF methods}\label{EBDFmethodssection}

The coefficients for the EBDF methods are listed, for example, in \cite{ruuthhundsdorfer2005}.

\subsubsection{The EBDF3 method}

For this method, the recursion \eqref{taundef} takes the form
\begin{equation}\label{EBDF3tau}
11 \tau_n  -18 \tau_{n-1}+9 \tau_{n-2}-2 \tau_{n-3}=0\quad (n\ge 4)
\end{equation}
with
\[
\tau_1={18}/{11},\quad \tau_2={126}/{121},\quad \tau_3={1212}/{1331}.
\]
We have $\tau_0=0$ and $n_0=1$, hence it is enough to prove 
$\tau_n>0$ for all $n\ge 1$.
One root of the characteristic polynomial corresponding to \eqref{EBDF3tau} is $1$, 
so we get the representation
\[
\tau_n=1+\left(\frac{7}{22}+\frac{i \sqrt{39}}{22}\right)^n+\left(\frac{7}{22}-\frac{i
   \sqrt{39}}{22}\right)^n \quad (n\ge 1).
\]
But for $n\ge 1$ we have
\[
\left|\frac{7}{22}+\frac{i \sqrt{39}}{22}\right|^n+
\left|\frac{7}{22}-\frac{i \sqrt{39}}{22}\right|^n=2\cdot \left(\frac{2}{11}\right)^{{n}/{2}}\le 9/10,
\]
and the positivity of $\tau_n>0$ follows.

\subsubsection{The EBDF4 method}

The recursion \eqref{taundef} now reads 
\[
25 \tau_n-48 \tau_{n-1}+36 \tau_{n-2}-16\tau_{n-3}+3\tau_{n-4}=0\quad (n\ge 5)
\]
with
\[
\tau_1={48}/{25},\quad \tau_2={504}/{625},\quad \tau_3={10992}/{15625},\quad
\tau_4={366516}/{390625}.
\]
We again have $\tau_0=0$ and $n_0=1$. The explicit form of the sequence is
\[
\tau_n=1+\sum_{j=1}^3 \varrho_{j}^n \quad (n\ge 1),
\]
where $\vr_1\in\mathbb{R}$ and $\vr_{2,3}\in\mathbb{C}\setminus\mathbb{R}$
are the three roots of the cubic polynomial $\{25, -23, 13, -3\}$.
This time we have for $n\ge 3$ that
\[
\sum_{j=1}^3 |\varrho_{j}|^n\le 3\cdot (3/5)^{n}\le 9/10,
\]
proving $\tau_n>0$ for $n\ge 1$.

\subsubsection{The EBDF5 method}

For this method, the recursion \eqref{taundef} is
\[
137 \tau_n-300 \tau_{n-1}+300\tau_{n-2}-200\tau_{n-3}+75\tau_{n-4}-12 \tau _{n-5}=0\quad (n\ge 6)
\]
with
\[
\tau_1={300}/{137},\quad \tau_2={7800}/{18769},\quad
\tau_3={1271400}/{2571353},
\]
\[
\tau_4={415574100}/{352275361},\quad
\tau_5={64978409160}/{48261724457}.
\]
We have $\tau_0=0$ and $n_0=1$. The explicit form of the sequence is
\[
\tau_n=1+\sum_{j=1}^4 \varrho_{j}^n \quad (n\ge 1),
\]
where $\vr_{1,2,3,4}\in\mathbb{C}\setminus\mathbb{R}$
are the four roots of the polynomial $\{137, -163, 137, -63, 12\}$.
But $|\vr_{1,2,3,4}|\le 71/100$, so for $n\ge 5$ we have
\[
\sum_{j=1}^4 |\varrho_{j}|^n\le 4\cdot \left({71}/{100}\right)^n\le 9/10,
\]
proving $\tau_n>0$ for $n\ge 1$.

\subsection{Proofs for the BDF methods}\label{BDFmethodssection}

The coefficients for the BDF methods are listed, for example, in \cite{hairerwanner}.

\subsubsection{The BDF1 method}

We include this method here for the sake of completeness. 
The recursion \eqref{genmunrec} now has the form
\[
(\gamma +1) \mu_n(\gam)-\mu_{n-1}(\gam)=0\quad\quad (n\ge 1)
\]
with
\[
\mu_0(\gam)=\frac{1}{\gam+1}.
\]
The explicit solution is $\mu_n(\gam)=1/(\gam+1)^{n+1}>0$, so,
due to Theorem \ref{thm1.1}, we have that $\gam$ is a $\scb$ for any $\gam>0$.

\subsubsection{The BDF2 method}

The recursion \eqref{genmunrec} takes the form
\[
(2 \gamma +3) \mu_n(\gam)-4 \mu_{n-1}(\gam)+\mu_{n-2}(\gam)=0\quad\quad (n\ge 2)
\]
with
\[
\mu_0(\gam)=\frac{2}{2 \gamma +3},\quad \mu_1(\gam)=\frac{8}{(2 \gamma +3)^2}.
\]
Its characteristic polynomial $\Pp_2(\cdot,\gam)$ is quadratic for $\gam>0$. 
This polynomial has
\begin{itemize}
\item two distinct real roots for $0<\gam<1/2$;
\item a double real root for $\gam=1/2$; 
\item a pair of complex conjugate roots for $\gam>1/2$.
\end{itemize}
For any fixed $\gam>1/2$ we thus have
\[
\mu_n(\gam)=|\vr_1(\gam)|^n\left[c_{1}(\gam)\left(\frac{\vr_1(\gam)}{|\vr_1(\gam)|}\right)^n+\overline{c_{1}(\gam)}\left(\frac{\overline{\vr_1(\gam)}}{|\vr_1(\gam)|}\right)^n\right]
\]
with a suitable $c_{1}(\gam)\in\mathbb{C}\setminus\{0\}$ and 
$\vr_1(\gam)\in\mathbb{C}\setminus\mathbb{R}$. Due to Lemma \ref{upperestlemma2}
with $\nu_n\equiv 0$, the expression in $[\ldots]$ is negative for infinitely many $n$.
Hence, by Theorem \ref{thm1.1}, $1/2+\varepsilon$ is not a $\scb$ for any $\varepsilon>0$,
implying $\gmtw<1/2+\varepsilon$.

Conversely, by verifying $\mu_n(1/2)=2^{-n-1} (n+1)\ge 0$ for all $n\in\mathbb{N}$
and taking into account \eqref{sect32intro}, we see that $\gmtw\ge 1/2$, so the proof is complete.

\subsubsection{The BDF3 method}

The recursion \eqref{genmunrec} is
\[
(6 \gamma +11) \mu_n(\gam)-18\mu_{n-1}(\gam)+9 \mu_{n-2}(\gam)-2 \mu_{n-3}(\gam)=0\quad\quad (n\ge 3)
\]
with
\[
\mu_0(\gam)=\frac{6}{6 \gamma +11},\quad \mu_1(\gam)=\frac{108}{(6 \gamma +11)^2},
\quad \mu_2(\gam)=\frac{54 (-6 \gamma+25)}{(6 \gamma +11)^3}.
\]

Let us consider the term
\[
\mu_6(\gam)=\frac{6 \left(5184 \gamma ^4-539352 \gamma ^3+4277340 \gamma ^2-7093698 \gamma
   +3248425\right)}{(6 \gamma +11)^7}.
\]
The polynomial $\{5184, -539352, 4277340, -7093698, 3248425\}$ 
in the numerator has 4 real roots; 
let $\gam^*\approx 0.831264$ denote its smallest root (the other 3 zeros are located at 
$\approx 1.22747$, $\approx 6.42689$ and $\approx 95.556$). Then, due to Lemma \ref{upperestlemma1}, we have $\gmth\le\gam^*$.

To complete the proof, we show that $\mu_n(\gam^*)\ge 0\  (\forall\,n\in\mathbb{N})$, meaning
that $\gmth\ge\gam^*$ by \eqref{sect32intro}.
Indeed, for $\gam=\gam^*$, the explicit form of the recursion is
\[
\mu_n(\gam^*)=c_1 \varrho_{1}^n+c_2 \varrho_{2}^n+\overline{c_2} (\overline{\varrho_{2}})^n
\quad\quad (n\ge 0),
\]
where 
\begin{itemize}
\item $\vr_1\approx 0.500518$ is the largest real root of the polynomial 
\begin{itemize}
\item[] $P_{\text{BDF31}}:=\{34012224, -85030560, 108650160, -91171656, 55033668, $
\item[] $-25076142, 8777889, -2366334, 486000, -75816, 10080, -1152, 64\}$;
\end{itemize}
\item $\vr_2\approx 0.312678 + 0.390087 i$ is the root of $P_{\text{BDF31}}$ with the largest
real part; 
\item $c_1\approx 0.50155509$ is the largest real root of the polynomial
\begin{itemize}
\item[] $P_{\text{BDF32}}:=\{91221089034315373632, -76017574195262811360,$
\item[] $26664298621295150160, -9975778735584785400, 2799915334883820972,$ 
\item[] $-498764709912473586, 93247136355378087, -8606361446997984,$
\item[] $425210419226880, -10041822761472, 76685377536, -237993984, 262144\}$;
\end{itemize}
\item $c_2\approx -0.0631319 - 0.270418 i$ is the root of $P_{\text{BDF32}}$ with the
smallest real part.
\end{itemize}
\begin{rem}\label{rem3612degree}
The 12th-degree algebraic numbers $\vr_{1,2}$ are of course roots of the cubic characteristic polynomial \eqref{Ppkcharpoly}, with $\gam$ replaced by the 4th-degree algebraic number $\gam^*$; 
that is, $\Pp_3(\vr_{1,2},\gam^*)=0$. 
\end{rem}
\begin{rem}
Notice that $|\vr_2|\approx 0.499935$ is relatively close 
to $|\vr_1|\approx 0.500518$. This results in an increased computational cost needed to finish 
the proof (cf. Remark \ref{remark32}).
\end{rem}

Now, clearly, $\mu_n(\gam^*)=\vr_1^n\left[c_1+c_2 \left({\varrho_{2}}/{\varrho_{1}}\right)^n+\overline{c_2} \left({\overline{\varrho_{2}}}/{\varrho_{1}}\right)^n\right]$, and we have
\[
\left| c_2 \left(\frac{\varrho_{2}}{\varrho_{1}}\right)^n+\overline{c_2} \left(\frac{\overline{\varrho_{2}}}{\varrho_{1}}\right)^n\right|\le 2 |c_2| \left|\frac{\varrho_{2}}{\varrho_{1}}\right|^n<
2\cdot\frac{2777}{10000} \left(\frac{9989}{10000}\right)^n.
\]
On the other hand,
\[
2\cdot\frac{2777}{10000} \left(\frac{9989}{10000}\right)^n<\frac{50155}{100\,000}<c_1
\]
for $n\ge 93$, therefore $\mu_n(\gam^*)>0$ for $n\ge 93$. 

Finally, one checks that $\mu_n(\gam^*)>0$ for $n\in\{0, 1, \ldots, 92\}\setminus\{6\}$
(recall that $\mu_6(\gam^*)=0$), so the proof is complete.
\begin{rem}
We have $\mu_{92}(\gam^*)\approx 1.585176\cdot 10^{-28}$.
\end{rem}

\subsubsection{The BDF4 method}\label{BDF4subsubsection}

The recursion \eqref{genmunrec} is
\[
(12 \gamma +25) \mu_n(\gam)-48 \mu_{n-1}(\gam)+36 \mu_{n-2}(\gam)-16 \mu_{n-3}(\gam)+
3 \mu_{n-4}(\gam)=0
\quad\quad (n\ge 4)
\]
with
\[
\mu_0(\gam)=\frac{12}{12 \gamma +25},\quad \mu_1(\gam)=\frac{576}{(12 \gamma +25)^2},
\quad \mu_2(\gam)=\frac{1296 (-4 \gamma+13 )}{(12 \gamma +25)^3},
\]
\[
\mu_3(\gam)=\frac{192 \left(144 \gamma ^2-1992 \gamma +2137\right)}{(12 \gamma +25)^4}.
\]

For $\gam>0$, the characteristic polynomial of the recursion, $\Pp_4(\cdot,\gam)$, 
has multiple roots if and only if $\gam=7/12\approx 0.5833$. In the rest of the proof,
it will be sufficient to focus on the interval $0<\gam<7/12$.

For any $0<\gam<7/12$, let us denote the four distinct roots of $\Pp_4(\cdot,\gam)$
by $\vr_{1,2,3,4}(\gam)$. Then $0<\vr_2(\gam)<\vr_1(\gam)<1$ and 
$\vr_{3,4}(\gam)\in\mathbb{C}\setminus\mathbb{R}$.
Let us denote by 
\begin{equation}\label{gamstar048622def}
\gam^*\approx 0.48622
\end{equation} 
the 5th-degree algebraic number
listed in the row of  $\gmfo$ in
Theorem \ref{thm2.2}. By separating the real and imaginary parts of
$\Pp_4(x+i y,\gam)$, then setting up and solving the
appropriate system of polynomial equations over the reals, we can prove that 
\begin{itemize}
\item $|\vr_{3}(\gam)|=|\vr_{4}(\gam)|<\vr_1(\gam)$ for $0<\gam<\gam^*$;
\item $|\vr_3(\gam^*)|=|\vr_4(\gam^*)|=\vr_1(\gam^*)$ for $\gam=\gam^*$;
\item $\vr_1(\gam)<|\vr_3(\gam)|=|\vr_4(\gam)|$ for $\gam^*<\gam<7/12$.
\end{itemize}
In other words, the positive real root $\vr_1(\gam)$ is no longer dominant for $\gam>\gam^*$.

First we prove $\gmfo\ge \gamma^*$ by proving $\mu_n(\gam^*)>0$ ($\forall\,n\in\mathbb{N}$), see \eqref{sect32intro}. For $\gam=\gam^*$ we have the representation
\[
\mu_n(\gam^*)=c_1(\gam^*) \left(\vr_1(\gam^*)\right)^n+c_2(\gam^*) \left(\vr_2(\gam^*)\right)^n+
c_3(\gam^*) \left(\vr_3(\gam^*)\right)^n+\overline{c_3(\gam^*)} 
\left(\overline{\vr_3(\gam^*)}\right)^n \quad (n\ge 0), 
\]
where 
\begin{itemize}
\item $\vr_1(\gam^*)\approx 0.605651$ is the unique real root of the polynomial 
\begin{itemize}
\item[] $\{96, -144, 86, -30, 9, -2\}$;
\end{itemize}
\item $\vr_2(\gam^*)\approx 0.437941$ is the unique real root of the polynomial 
\begin{itemize}
\item[] $\{7080, -8928, 6410, -2826, 621, -54\}$;
\end{itemize}
\item $\vr_3(\gam^*)\approx 0.25655 + 0.54863 i$ is the root of the polynomial 
\begin{itemize}
\item[] $\{82944, -140544, 160624, -112944, 60516, -27800, 12636, -5832, 1969, 
-384, 36\}$
\end{itemize}
having the property $|\vr_3(\gam^*)|=\vr_1(\gam^*)$;
\item $c_1(\gam^*)\approx 1.21912$ is the unique real root of the polynomial 
\begin{itemize}
\item[] $\{638976, -1308672, 767680, -148848, 255, -16\}$;
\end{itemize}
\item $c_2(\gam^*)\approx -0.583734$ is the unique real root of the polynomial 
\begin{itemize}
\item[] $\{15582086307840, 11032756568064, 1374924543424, 141329286000,$ 
\item[] $\ -715299903, 8503056\}$;
\end{itemize}
\item $c_3(\gam^*)\approx -0.123106 - 0.169757 i$ is the root of the polynomial 
\begin{itemize}
\item[] $\{654252399875063808, 147972616215330816, 117436085430648832,$ 
\item[] $\ 23378947275620352, 6522272391303168, 504776558675968, 75411131715456,$ 
\item[] $\ -3364763918784, 58367021905, -452933856, 1679616\}$
\end{itemize}
having the smallest real part.
\end{itemize}
\begin{rem}
Here again we have converted polynomials whose coefficients are algebraic numbers to
higher-degree polynomials with integer coefficients (cf. Remark \ref{rem3612degree}).
\end{rem}

By rewriting $\mu_n(\gam^*)$ ($n\in\mathbb{N}$) as
\[
\left(\vr_1(\gam^*)\right)^n\left[c_1(\gam^*) +c_2(\gam^*) \left(\frac{\vr_2(\gam^*)}{\vr_1(\gam^*)}\right)^n+
c_3(\gam^*) \left(\frac{\vr_3(\gam^*)}{\vr_1(\gam^*)}\right)^n+\overline{c_3(\gam^*)} 
\left(\frac{\overline{\vr_3(\gam^*)}}{\vr_1(\gam^*)}\right)^n\right], 
\]
and noticing that
\[
\left| c_2(\gam^*) \left(\frac{\vr_2(\gam^*)}{\vr_1(\gam^*)}\right)^n+
c_3(\gam^*) \left(\frac{\vr_3(\gam^*)}{\vr_1(\gam^*)}\right)^n+\overline{c_3(\gam^*)} 
\left(\frac{\overline{\vr_3(\gam^*)}}{\vr_1(\gam^*)}\right)^n\right|\le
\]
\[
|c_2(\gam^*)| \left|\frac{\vr_2(\gam^*)}{\vr_1(\gam^*)}\right|^n+
2|c_3(\gam^*)| \left|\frac{\vr_3(\gam^*)}{\vr_1(\gam^*)}\right|^n=
|c_2(\gam^*)| \left|\frac{\vr_2(\gam^*)}{\vr_1(\gam^*)}\right|^n+
2|c_3(\gam^*)|< 
\]
\[
|c_2(\gam^*)| + 2|c_3(\gam^*)|< 11/10 <|c_1(\gam^*)|,
\]
we see that $\mu_n(\gam^*)>0$ for all $n\in\mathbb{N}$. Thus we have proved $\gmfo\ge \gamma^*$.

To prove the converse inequality, $\gmfo\le \gamma^*$, we apply Lemma \ref{upperestlemma2}. 
We set $\gam:=\gam^*+\varepsilon$ with some sufficiently small, but arbitrary $\varepsilon>0$. Then for $n\in\mathbb{N}$
we have 
\[
\mu_n(\gam)=\left|\vr_3(\gam)\right|^n\left( \nu_n+
w z^n+\bar{w} (\bar{z})^n\right)
\]
with $z:={\vr_3(\gam)}/{|\vr_3(\gam)|}$, $w:=c_3(\gam)$ and 
\[
\nu_n:=c_1(\gam) \left(\frac{\vr_1(\gam)}{|\vr_3(\gam)|}\right)^n+c_2(\gam) \left(\frac{\vr_2(\gam)}{|\vr_3(\gam)|}\right)^n.
\]
Due to the properties of the numbers $\vr_j(\gam)$ listed in the paragraph of \eqref{gamstar048622def}, we know that $\mathbb{R}\ni \nu_n\to 0$ as $n\to+\infty$.
Moreover, since the functions $\vr_3(\cdot)$ and $c_3(\cdot)$ are continuous (also) at $\gam^*$, 
we have $z\in\mathbb{C}\setminus\mathbb{R}$, $|z|=1$ and $w\in\mathbb{C}\setminus\{0\}$, 
for $\varepsilon>0$ small enough.
Lemma \ref{upperestlemma2} then shows that $\mu_n(\gam)$ cannot be non-negative
for all $n\in\mathbb{N}$, so by Theorem \ref{thm1.1} we obtain that 
$\gmfo< \gamma^*+
\varepsilon$.

\subsubsection{The BDF5 method}

The recursion \eqref{genmunrec} is
\[
(60 \gamma +137) \mu_n(\gam)-300 \mu_{n-1}(\gam)+300 \mu_{n-2}(\gam)
-200 \mu_{n-3}(\gam)+
\]
\[
75 \mu_{n-4}(\gam)-12 \mu_{n-5}(\gam)=0\quad\quad (n\ge 5)
\]
with
\[
\mu_0(\gam)=\frac{60}{60 \gamma +137},\quad \mu_1(\gam)=\frac{18000}{(60 \gamma +137)^2},
\quad \mu_2(\gam)=\frac{18000 (-60 \gamma+163 )}{(60 \gamma +137)^3},
\]
\[
\mu_3(\gam)=\frac{12000 \left(3600 \gamma ^2-37560 \gamma +30469\right)}{(60 \gamma +137)^4},
\]
\[
\mu_4(\gam)=\frac{4500 \left(-216000 \gamma ^3+8600400 \gamma ^2-22146420 \gamma
   +10021847\right)}{(60 \gamma +137)^5};
\]
see Figure \ref{BDF5fig}.

The characteristic polynomial of the recursion $\Pp_5(\cdot,\gam)$ has no multiple roots
for $\gam>0$. We denote the five distinct roots of $\Pp_5(\cdot,\gam)$
by $\vr_{1,2,3,4,5}(\gam)$ and let 
\[
\gam^*\approx 0.30421
\]
denote the 10th-degree algebraic number
listed in the row of  $\gmfi$ in
Theorem \ref{thm2.2}. Then for any $\gam\in(0,1)$ we can prove that
\begin{itemize}
\item $0<\vr_1(\gam)<1$ and $\vr_{2,3,4,5}(\gam)\in\mathbb{C}\setminus\mathbb{R}$;
\item $|\vr_{2,3,4,5}(\gam)|<\vr_1(\gam)$ for $0<\gam<\gam^*$;
\item $|\vr_{4,5}(\gam^*)|<|\vr_{2,3}(\gam^*)|=\vr_1(\gam^*)$ for $\gam=\gam^*$;
\item $|\vr_{1,4,5}(\gam)|<|\vr_2(\gam)|=|\vr_3(\gam)|$ for $\gam^*<\gam<1$.
\end{itemize}
\begin{figure}
\centerline{\includegraphics[width=.75\textwidth]{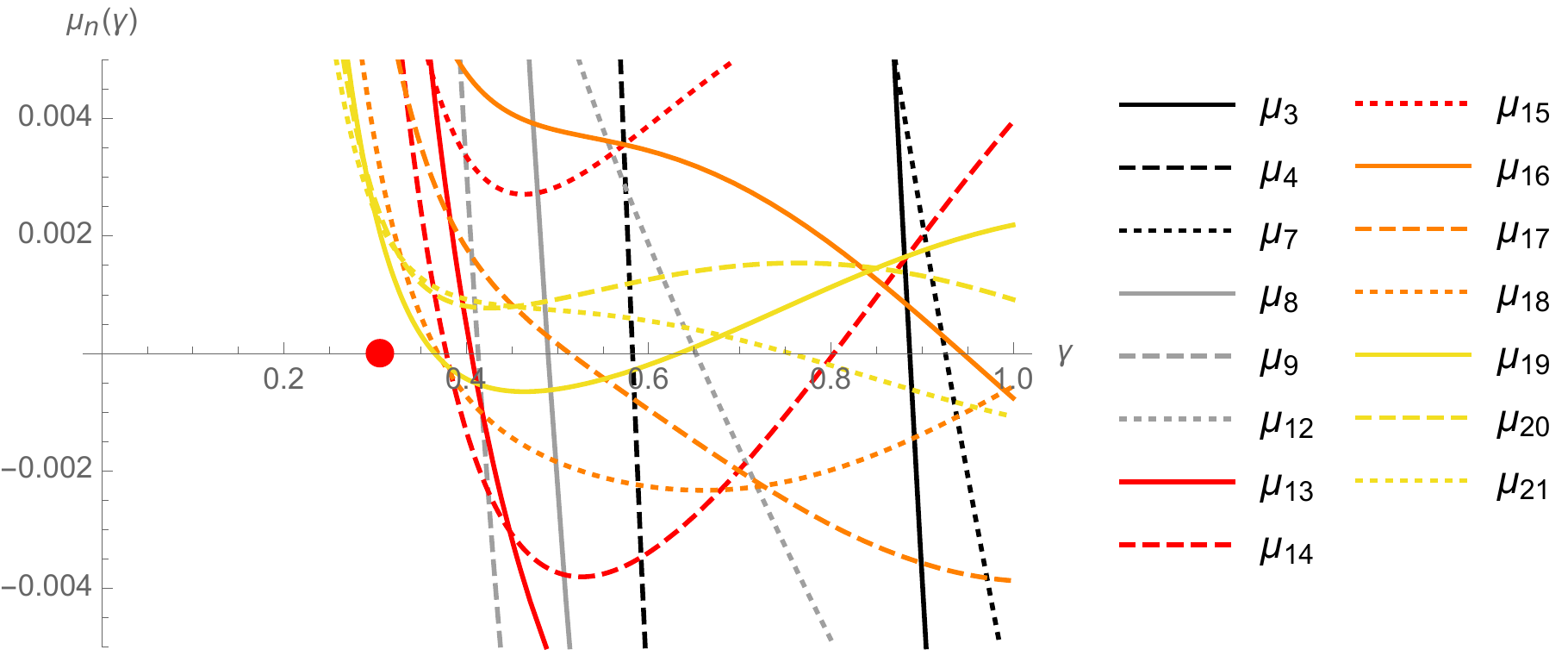}}
\caption{The functions $\gam\mapsto\mu_n(\gam)$ for $1\le n\le 21$ corresponding to the BDF5 method
are shown (the curves with indices $n\in\{1,2,5,6,10,11\}$ are not visible in this plot window). The
red dot is placed at $\gam=\gmfi\approx 0.30421$.}
\label{BDF5fig}      
\end{figure}
For $\gam=\gam^*$ and $n\ge 0$ we have 
\[
\mu_n(\gam^*)=\varrho_{1}^n\left[ c_1 +c_2 \left(\frac{\varrho_{2}}{\vr_1}\right)^n+\overline{c_2} \left(\frac{\overline{\varrho_{2}}}{\vr_1}\right)^n+c_4 \left(\frac{\varrho_{4}}{\vr_1}\right)^n+\overline{c_4} \left(\frac{\overline{\varrho_{4}}}{\vr_1}\right)^n\right],
\]
where, for brevity, now we omit the explicit form of the algebraic numbers $c_j$ and $\vr_j$,
and only give their approximate values:
\begin{itemize}
\item $\vr_1\approx 0.737893$,\quad $\vr_2\approx 0.195442 + 0.711539 i$,\quad
$\vr_4\approx 0.401777 + 0.175943 i$,
\item $c_1\approx 0.994377$,\quad $c_2\approx -0.117157 - 0.126015 i$,\quad 
$c_4\approx -0.186798 - 0.0841337 i$.
\end{itemize}
From this we get that
\[
\left| c_2 \left(\frac{\varrho_{2}}{\vr_1}\right)^n+\overline{c_2} \left(\frac{\overline{\varrho_{2}}}{\vr_1}\right)^n+c_4 \left(\frac{\varrho_{4}}{\vr_1}\right)^n+\overline{c_4} \left(\frac{\overline{\varrho_{4}}}{\vr_1}\right)^n\right|\le 2|c_2|\cdot 1^n+ 2|c_4|\cdot 
1^n< \frac{8}{10}<c_1,
\]
so $\mu_n(\gam^*)>0$ ($n\in\mathbb{N}$), and hence $\gmfi\ge \gamma^*$ by
\eqref{sect32intro}.

The proof of the converse inequality, $\gmfi\le \gamma^*$, is again based on Lemma \ref{upperestlemma2}, and is completely analogous to the 
one presented in Section \ref{BDF4subsubsection}.

\subsubsection{The BDF6 method}

The recursion \eqref{genmunrec} is
\[
3 (20 \gamma +49) \mu_n(\gam)-360 \mu_{n-1}(\gam)
+450\mu_{n-2}(\gam)-400\mu_{n-3}(\gam)+
225\mu_{n-4}(\gam)-
\]
\[
72\mu_{n-5}(\gam)+10\mu_{n-6}(\gam)=0 \quad\quad (n\ge 6)
\]
with
\[
\mu_0(\gam)=\frac{20}{20 \gamma +49},\quad \mu_1(\gam)=\frac{2400}{(20 \gamma +49)^2},
\]
\[
\mu_2(\gam)=\frac{3000 (-20 \gamma+47 )}{(20 \gamma +49)^3},\quad \mu_3(\gam)=\frac{8000 \left(400 \gamma ^2-3440 \gamma +2131\right)}{3 (20 \gamma +49)^4},
\]
\[
\mu_4(\gam)=\frac{500 \left(-24000 \gamma ^3+695600 \gamma ^2-1343380 \gamma +474833\right)}{(20\gamma +49)^5},
\]
\[
\mu_5(\gam)=\frac{160 \left(480000 \gamma ^4-53296000 \gamma ^3+283987200 \gamma ^2-212499240 \gamma
   +84071653\right)}{(20 \gamma +49)^6}.
\]

Let us consider any $0\le \gam<37/60\approx 0.6167$. Then one checks by using 
the discriminant that the 6 roots of $\Pp_6(\cdot,\gam)=0$ are distinct.
\begin{figure}
\centerline{\includegraphics[width=.6\textwidth]{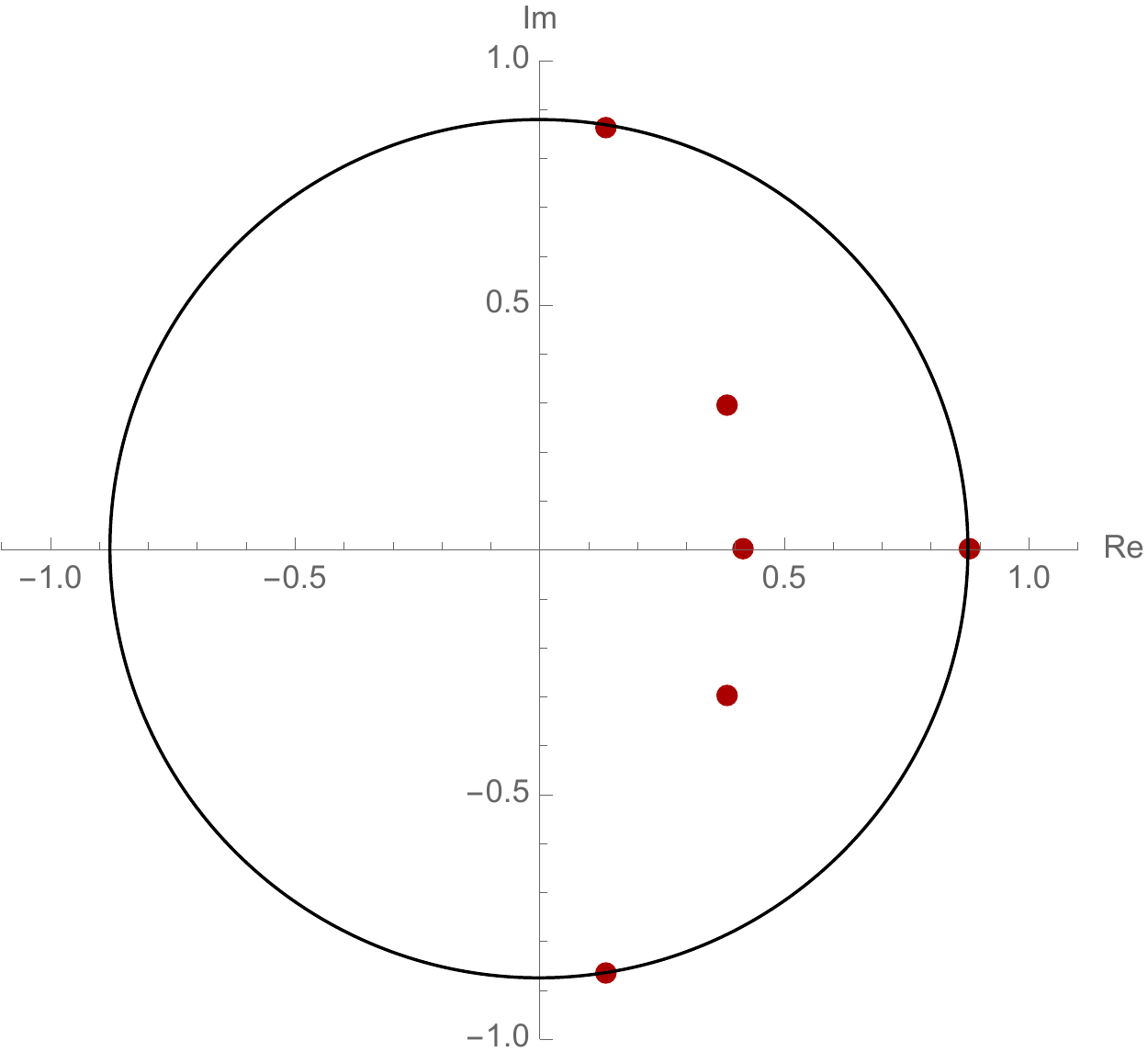}}
\caption{The roots of $\Pp_6(\cdot,\gmsi)$ corresponding to the BDF6 case.}
\label{BDF6fig}      
\end{figure}
Let 
\[
\gam^*\approx 0.13136
\]
denote the 18th-degree algebraic number
listed in the row of  $\gmsi$ in Theorem \ref{thm2.2}. This constant has been obtained after some non-trivial computation and simplification.
The roots $\vr_j(\gam)$ ($1\le j\le 6$) are distributed as follows:
\begin{itemize}
\item $0<\vr_2(\gam)<\vr_1(\gam)<1$ and $\vr_{3,4,5,6}(\gam)\in\mathbb{C}\setminus\mathbb{R}$;
\item $|\vr_{2,3,4,5,6}(\gam)|<\vr_1(\gam)$ for $0<\gam<\gam^*$;
\item $|\vr_{2,5,6}(\gam^*)|<|\vr_{3,4}(\gam^*)|=\vr_1(\gam^*)$ for $\gam=\gam^*$;
\item $|\vr_{1,2,5,6}(\gam)|<|\vr_3(\gam)|=|\vr_4(\gam)|$ for $\gam^*<\gam<37/60$.
\end{itemize}
For $\gam=\gam^*$ and $n\ge 0$, one has the representation 
\[
\mu_n(\gam^*)=\varrho_{1}^n\left[ c_1 +c_2 \left(\frac{\varrho_{2}}{\vr_1}\right)^n
+c_3 \left(\frac{\varrho_{3}}{\vr_1}\right)^n+\overline{c_3} \left(\frac{\overline{\varrho_{3}}}{\vr_1}\right)^n+c_5 \left(\frac{\varrho_{5}}{\vr_1}\right)^n+\overline{c_5} \left(\frac{\overline{\varrho_{5}}}{\vr_1}\right)^n\right],
\]
where the algebraic numbers $c_j$, $\vr_j$ have the approximate values
\begin{itemize}
\item $\vr_1\approx 0.87690236$,\quad $\vr_2\approx 0.41284041$,
\item $\vr_3\approx  0.13673253 + 0.86617664 i$,\quad $\vr_5\approx 0.38057439 + 0.29512217 i$,
\item $c_1\approx 1.0000077$,\quad $c_2\approx -0.13742979$,
\item $c_3\approx -0.11295491 - 0.10160183 i$,\quad  
$c_5\approx -0.124637633 - 0.050848744 i$,
\end{itemize}
see Figure \ref{BDF6fig}. 
\begin{figure}
\centerline{\includegraphics[width=.7\textwidth]{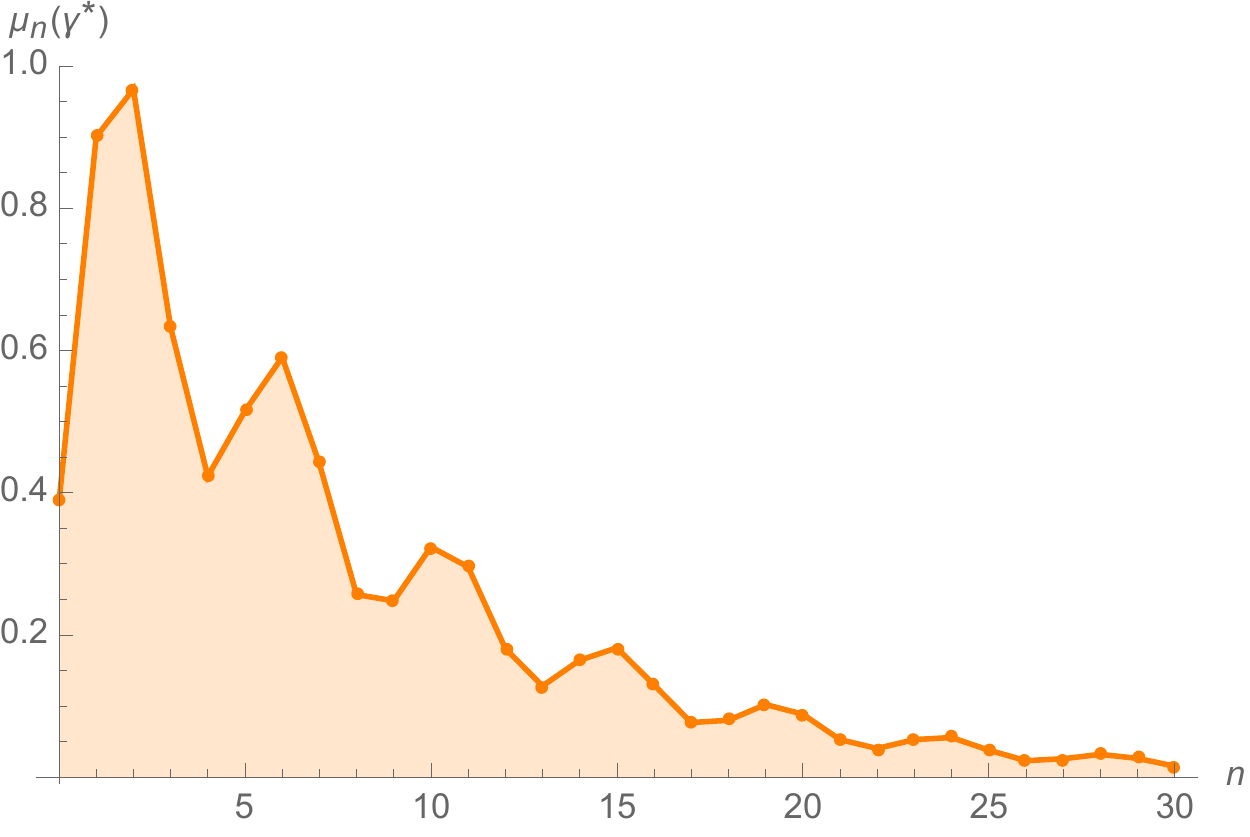}}
\caption{The sequence $\mu_n(\gam^*)$ 
 corresponding to the BDF6 method
is depicted (using linear interpolation).}
\label{BDF6mun}      
\end{figure}
For any $n\ge 0$, the estimate
\[
\left|c_2 \left(\frac{\varrho_{2}}{\vr_1}\right)^n
+c_3 \left(\frac{\varrho_{3}}{\vr_1}\right)^n+\overline{c_3} \left(\frac{\overline{\varrho_{3}}}{\vr_1}\right)^n+c_5 \left(\frac{\varrho_{5}}{\vr_1}\right)^n+\overline{c_5} \left(\frac{\overline{\varrho_{5}}}{\vr_1}\right)^n\right|\le
\]
\[
|c_2|+2|c_3|+2|c_5|<\frac{8}{10}<c_1
\]
yields $\mu_n(\gam^*)>0$, see Figure \ref{BDF6mun}. This proves that $\gmsi\ge \gamma^*$ by 
\eqref{sect32intro}.

As before, a final application of Lemma \ref{upperestlemma2} shows that $\gmsi\le \gam^*$,
 completing the proof.

\subsection{Proofs for the Adams--Bashforth methods}\label{ABmethodssection}

The coefficients for the Adams--Bashforth methods are listed, for example, in \cite{hairerwanner1}.

\subsubsection{The AB1 method}

It is easily seen that the recursion \eqref{mundef} now has the form
\[
\mu_n(\gam)=(1-\gam) \mu_{n-1}(\gam)\quad\quad (n\ge 2)
\]
with $\mu_1(\gam)=1$, so any $\gam>1$ violates the non-negativity of $\mu_n(\gam)$
in \eqref{necsufcond2}. Hence $\gmo\le 1$. But $\mu_n(1)\ge 0$ for all $n\in\mathbb{N}$,
and it is known \cite{hairerwanner} that $-1\in \mathrm{int}(\Ss)$, so
\eqref{17mod} finishes the proof.

\subsubsection{The AB2 method}

For this method, the recursion \eqref{mundef} is
\[
\mu_n(\gam)-\left(1-\frac{3 \gamma }{2}\right) \mu_{n-1}(\gam)-\frac{\gam}{2}\mu_{n-2}(\gam)=0
\quad\quad (n\ge 3)
\]
with $\mu_1(\gam)=3/2$ and $\mu_2(\gam)=-9\gam/4+1$. Lemma \ref{upperestlemma1}
with $n=2$ and $\gam^*=4/9$ shows that $\gmtw\le 4/9$. On the other hand,
\[
\mu_n(4/9)=3^{1-n} \left(2^n-4 (-1)^n\right)/4\ge 0\quad\quad (n\ge 1),
\]
and $-4/9\in \mathrm{int}(\Ss)$ (see \cite{hairerwanner}), so the proof is complete due to
\eqref{17mod}.

\subsubsection{The AB3 method}

For this method, the recursion \eqref{mundef} takes the form
\[
\mu_n(\gam)-\left(1-\frac{23 \gam}{12}\right) \mu_{n-1}(\gam)
-\frac{4}{3} \gam  \mu_{n-2}(\gam)+\frac{5}{12} \gam  \mu_{n-3}(\gam)=0\quad\quad (n\ge 4)
\]
with
\[
\mu_1(\gam)=\frac{23}{12},\quad \mu_2(\gam)=-\frac{529 \gamma }{144}+\frac{7}{12},\quad \mu_3(\gam)=\frac{12167 \gamma ^2}{1728}-\frac{161 \gamma }{72}+1.
\]
Lemma \ref{upperestlemma1}
with $n=2$ and $\gam^*={84}/{529}$ shows that $\gmth\le {84}/{529}$. 
We also know \cite{hairerwanner} that $-{84}/{529}\in \mathrm{int}(\Ss)$, so
by \eqref{17mod} it is enough to verify that $\mu_n({84}/{529})\ge 0$ for all $n\ge 1$.

For $n\ge 1$ we have
\[\mu_n({84}/{529})=\sum_{j=1}^3 c_j \varrho_{j}^n=\varrho_{3}^n \left(c_3+
c_1\left(\frac{\varrho_1}{\varrho_3}\right)^n+c_2\left(\frac{\varrho_2}{\varrho_3}\right)^n\right),\]
where the numbers $\varrho_j$ ($\varrho_1<0<\varrho_2<\varrho_3$, $|\varrho_{1}|<\varrho_3/2$, 
$\varrho_{2}<\varrho_3/4$) and 
$c_j$ ($-5<c_1<-3<c_2<0<1<c_3$)
are the three roots of the polynomials
$\{529, -368, -112, 35\}$ and 
\[
\{30733417008, 193547352348, 162435667337, -391554926405\},
\]
respectively. Since 
\[
\left|c_1\left(\frac{\varrho_1}{\varrho_3}\right)^n+c_2\left(\frac{\varrho_2}{\varrho_3}\right)^n\right|
\le 5 \left|\frac{\varrho_1}{\varrho_3}\right|^n+3 \left|\frac{\varrho_2}{\varrho_3}\right|^n
<\frac{5}{2^n}+\frac{3}{4^n}<1<c_3
\]
for $n\ge 3$, and $\mu_n({84}/{529})\ge 0$ for $1\le n\le 2$,
the proof is complete.

\subsubsection{The AB4 method}

The starting terms of the recursion \eqref{mundef} satisfy
\[
\mu_1(\gam)=\frac{55}{24},\quad \mu_2(\gam)=-\frac{3025 \gamma }{576}-\frac{1}{6},
\]
so the non-negativity condition in \eqref{necsufcond2} for $n=2$ is violated for any $\gam>0$,
hence $\noscb$.

\section{Conclusions}\label{conclusionssection}

The step-size coefficient for boundedness (SCB) of a linear multistep method (LMM) is a generalization of the 
 strong-stability-preserving (SSP) coefficient of the LMM. The SCB appears in conditions that ensure monotonicity or boundedness properties of the LMM, and a method is more efficient if it possesses a larger SCB. 
 
In \cite{hundsdorferspijkermozartova2012,spijker2013}, a necessary and sufficient condition
has been given for a  number $\gamma>0$ to be a SCB of a LMM. This condition involves checking the non-negativity of an auxiliary sequence $\mu_n(\gamma)$ that satisfies a linear recurrence relation in $n\in\mathbb{N}$. For fixed $n$, the function 
$\mu_n(\cdot)$ is a rational function.

The main goal of the present work is to determine the maximum SCB,  $\gm$ for a given linear multistep method.
For each $k$-step BDF method ($2\le k\le 6$) and each $k$-step Adams--Bashforth method ($1\le k\le 3$), 
we determine the exact value of  $\gm$ by finding the largest $\gamma>0$ that satisfies
$\mu_n(\gamma)\ge 0$ for all non-negative $n$. 

We have identified two types of conditions that characterize
  $\gm$ in these multistep families:\\
\indent ($i$) a positive real dominant root of the characteristic polynomial corresponding to the recursion $\mu_n(\gamma)$ loses its dominant property at $\gamma=\gm$, or\\
\indent ($ii$) there is an index $n_0\in\mathbb{N}$ such that $\gm$ is a simple root of the function $\mu_{n_0}(\cdot)$.\\
It turns out that $\gm$ is determined 
\begin{itemize}
\item by condition ($i$) for the BDF methods with $k\in\{2, 4, 5, 6\}$ steps;
\item by condition ($ii$) with $n_0=6$ for the $3$-step BDF method;
\item by condition ($ii$) with $n_0=2$ for the Adams--Bashforth methods with $k\in\{1, 2, 3\}$ steps.
\end{itemize}

\bigskip

\noindent \textbf{Acknowledgements.} The author is indebted to the anonymous referees of the manuscript for their suggestions that helped improving the presentation of the material.

\end{document}